\documentclass[11pt]{article}
\usepackage{amsthm, amsmath, amssymb, mathrsfs, mathtools}
\usepackage[all]{xy}
\usepackage{fancyhdr}
\usepackage{textcomp}
\usepackage{tabto}
\usepackage{biblatex}
\author{Skyler Marks}
	\newcommand{\xto}{\xrightarrow}
	\DeclareMathOperator{\Kom}{\textbf{CCH}}
	\DeclareMathOperator{\Ext}{Ext}
	\DeclareMathOperator{\cok}{cok}
	\DeclareMathOperator{\Cok}{Cok}
	\DeclareMathOperator{\Ker}{Ker}
	\DeclareMathOperator{\im}{im}
	\DeclareMathOperator{\ima}{Im}
	\DeclareMathOperator{\coim}{coim}
	
	\renewcommand{\hom}{\mathrm{Hom}}

	\renewcommand{\[}{\begin{equation}}
	\renewcommand{\]}{\end{equation}}

	\makeatletter
\def\tagform@#1{\maketag@@@{\ignorespaces(#1\unskip\@@italiccorr)}}
\usepackage[letterpaper, top=3cm, inner=3.3cm, outer=3.3cm]{geometry}
\makeatother
\newtheorem{theorem}{Theorem}[section]
\newtheorem{corollary}{Corollary}[theorem]
\newtheorem{lemma}[theorem]{Lemma}

\newtheorem{definition}{Definition}[section]

\nonstopmode

\title{A Categorical Development of Right Derived Functors}
\author{Skyler Marks\thanks{Boston University}}

\begin{document}
\maketitle

\begin{abstract}
	Category theory is the language of homological algebra, allowing us to state
	broadly applicable theorems and results without needing to
	specify the details for every instance of analogous objects.
	However, authors often stray from the realm of pure abstract category
	theory in their development of the field, leveraging the Freyd-Mitchell
	embedding theorem or similar results, or otherwise using set-theoretic
	language to augment a general categorical discussion. This paper seeks to
	demonstrate that - while it is not necessary for most mathematicians'
	purposes - a development of homological concepts can be contrived from
	purely categorical notions. We begin by outlining the categories we will
	work within, namely Abelian categories (building off additive categories).
	We continue to develop cohomology groups of sequences, eventually
	culminating in a development of right derived functors. This paper is
	designed to be a
	minimalist construction, supplying no examples or motivation beyond what is
	necessary to develop the ideas presented. 
\end{abstract}

\pagebreak
\section{Categorical Background}
\subsection{Set - Theoretic Preface}
In order to leverage the full power of category theory, we endeavor to work in a
general category with no set-theoretic language. As such, we will not restrict
ourselves to working within a small category; the classes of objects and
morphisms within the
category may be a proper class. However, we will need some machinery to work
with these classes. As such, the axioms of
Zermelo - Fraenkel + choice set theory (\textbf{ZFC}) will make for a clumsy
formalism for our purposes (as a general class is not treated). To rectify this
issue, we turn
to the conservative extension of \textbf{ZFC},  von Neumann - Bernays - Gödel
set theory (\textbf{NBG}), for our treatment of classes. More detail can be found
in \autocite{men}.  
A category, for the purpose of this paper, will be thought of as a (possibly proper) class
of objects together with (possibly proper) classes of morphisms for each pair of
objects (pairing allows for this, as each object can be represented as a set). Most categories of algebraic
objects are not small (or even essentially small) categories; thus, although
considering general classes is
often irrelevant to a mathematician's daily work, it is not entirely without
justification.
\subsection{Additive Categories}
The additive category forms the foundation for abelian categories, which in turn
comprise the
realm in which general homological algebra is performed. There are
many formulations of an additive category, most of which specify outright that the
classes $\hom(A, B)$ (morphisms between objects $A$ and $B$) should be abelian groups.
This amounts to adding a structure to the category, which (although easy to work
with) is somewhat undesirable. Our definition shies away
from this measure, instead characterizing additive categories as ones which
include certain objects. These certain objects will allow us to \textit{induce} 
an abelian group on the hom-classes of an additive category, without
specifying any additional structure.

Perhaps the most central of these objects is a
\textbf{biproduct} for any two elements $A$ and $B$, generalizing the notion of
a direct sum of abelian groups or the carteisan product of two sets to an arbitrary category.
Such a set theoretic product 
is generally a set of pairs $(a, b)$, with $a$ belonging to the first set of the product
and $b$ belonging to the second. We generalize this notion to a categorical
framework, using
morphisms which project off the product (effectively mapping $(a, b)$ to either
$a$ or $b$). Then, as is standard with category theory, we can dualize this
definition. 
\begin{definition}
	The \textbf{binary product} $A\times B$ in a category $\mathscr A$ is an object in
	$\mathscr A$ together with morphisms $\pi_1: A\times B\to A$ and $\pi_2: A\times B\to
	B$ such that for any object $C$ in $\mathscr A$ with maps $f:C\to A$ and $f':C\to B$
	there is a unique map $g:C\to A\times B$ such that $\pi_1 \circ g = f$ and
	$\pi_2 \circ g = f'$. The \textbf{binary coproduct} $A\amalg B$ is the dual
	of this; an object in $\mathscr A$ with maps $\iota_1: A\to A\amalg B$ and
	$\iota_2:B\to A\amalg B$ such
	that for any object $C$ in $\mathscr A$ with maps $f:A\to C$ and $f':B\to C$
	there is a unique map $g: A\amalg B\to C$ such that $g\circ\iota_1 = f$ and
	$g\circ \iota_2 = f'$.
\end{definition}

Some examination shows that the coproduct in the category of sets is the
disjoint union, and that the coproduct in the category of groups is the internal
direct product. Of note, however, is that in the category of groups the internal
direct sum is isomorphic to the external direct sum; speaking
categorically, the product is also a coproduct. We generalize this notion as a
\textbf{biproduct}, which will form the foundation for our construction of
abelian categories. 

\begin{definition}
	For objects $A$ and $B$ in a category $\mathscr A$ a \textbf{biproduct} is
	an object of $\mathscr A$ which is both a product and a coproduct with
	\[\pi_1\circ \iota_1 = 1_A\mathrm{\ and\ }\pi_2\circ \iota_2 =
	1_B\label{prod_ident}\]
\end{definition}

We will use this object to construct a "sum" of two morphisms. The associativity
of the product, crucial to the proof that the constructed sum forms a group, is
given by a repeated application of the definition; for more detail, see \autocite{maclane}. We
record this as a lemma for reference later.
\begin{lemma}\label{prod_assoc}
	The product $(A\oplus B)\oplus C$ is naturally isomorphic to $A\oplus
	(B\oplus C)$.
\end{lemma}

In order to leverage this this construction towards a sum of two morphisms, we
demonstrate the existence of a "biproduct of morphisms" with a construction
generalizing the
notion of the homomorphism $(x, y) \mapsto (f(x), g(y))$ mapping between
direct sums of groups, for homomorphisms $f$ and $g$. Speaking in categorical
terms, the only way to refer to the "components"
of the product is by use of the canonical projections; thus, we
accomplish this generalization by constructing a morphism which commutes with
the canonical projections off the biproducts.

\begin{lemma}\label{fg}
	For each pair for each pair
	of morphism $f:A\to B$ and $g:A'\to B'$, there is a unique morphism
	$(f\oplus
	g):A\oplus A'\to B\oplus B'$ such that $\pi_1\circ (f\oplus g) = f\circ p_1$ and $\pi_2\circ (f\oplus g)
	= g\circ p_2$,
	where $\pi_1$ and $\pi_2$ are the canonical projections from $B\oplus B'$ to
	$B$ and $B'$ respectively, and $p_1, p_2$ are the projections from $A\oplus
	A'$ to $A$ and $A'$ respectively, all as given by the
	definition of a categorical biproduct.
\end{lemma}
\begin{proof}
	The product $B\oplus B'$ projects onto $B$ and $B'$, and is universal with
	respect to that property. That is, for any other object $C$ and any other
	morphism $\pi'_1:C\to B, \pi'_2:C\to B'$ there must be a unique morphism $\phi: C\to B\oplus B$
	such that $\pi'_1 = \pi_1\circ\phi$ and $\pi'_2 = \pi_2\circ\phi$. Letting $C=
	A\oplus A'$, $\pi'_1 = f\circ p_1$ and $\pi'_2 = g\circ p_2$ 
	gives us a unique morphism $\phi = (f\oplus g): A\oplus A'\to B\oplus B'$ such that
	$\pi_1\circ(f\oplus g) = f\circ p_1$, and $\pi_2\circ (f\oplus g) = g\circ p_2$, as intended. 
\end{proof}
To leverage the duality inherent within categorical proofs, we demonstrate that
this morphism also commutes with inclusions.
\begin{lemma}
	The morphism $f\oplus g$ also satisfies $f\oplus g\circ i_1 = \iota_1\circ
	f$ and $f\oplus g\circ i_2 = \iota_2\circ g$, where $i_1$ and $i_2$ are the
	inclusions into $A\oplus A$ and $\iota_1$ and $\iota_2$ the inclusions into
	$B\oplus B$.\label{fg'}
\end{lemma}
\begin{proof}
	The morphism $f\oplus g$ is the unique morphism satisfying $\pi_1\circ
	(f\oplus g) = f\circ p_1$ and $\pi_2\circ (f\oplus g) = f\circ p_2$, as
	above. However, if we apply an inclusion on the left and right of each of
	these equations we obtain:
	$$\iota_1\circ\pi_1\circ(f\oplus g)\circ i_1 = \iota_1\circ f\circ p_1\circ i_1
	\text{ and }
	\iota_2\circ\pi_2\circ(f\oplus g)\circ i_2 = \iota_2\circ f\circ p_2\circ i_2$$
	Which simplify by \eqref{prod_ident} to 
	$$(f\oplus g)\circ i_1 = \iota_1\circ f
	\text{ and }
	(f\oplus g)\circ i_2 = \iota_2\circ f$$
	As intended. Thus $f\oplus g$ must be the unique morphism (with uniqueness
	given by an argument dual to lemma \ref{fg}) satisfying this property.
\end{proof}
The final constructions the biproduct contributes are the diagonal map,
generalizing a group homomorphism $x\mapsto (x, x)$, and its dual the codiagonal.
The codiagonal is especially powerful, effectively generalizing the notion of
addition. Note, especially, that existence of these maps follows from the
existence of the biproduct.
\begin{definition}
	The \textbf{diagonal map} $d: A\to A\oplus A$ is the unique morphism given
	by substituting the $A$ and the identity maps $1_A:A\to A$ into the categorical
	definition of a product, such that $\pi_1\circ d = \pi_2\circ d = 1_A$. The
	\textbf{codiagonal map} $d':A\amalg A\to A$ is the dual of this notion for a
	coproduct, with $d'\circ\iota_1 = d'\circ\iota_2 = 1_A$.
\end{definition}
These constructions, together with a few standard categorical definitions, allow
us to state the definition of an additive category.
\begin{definition}\label{add}
	An \textbf{additive} category $\mathscr{A}$ satisfies three axioms:
	\begin{enumerate}
		\item The category $\mathscr A$ has an object which is both initial and
			final, called a zero object or 0. A morphism which factors through
			the zero object is called a zero morphism, often denoted 0 as well
			by abuse of notation.
		\item For each pair of objects $X, Y$ in $\mathscr A$, there is a
			biproduct $X\oplus Y$ which satisfies the definitions for both
			products and coproducts in such a way that cannonical projections
			commute with cannonical inclusions. This condition
			allows us to define an operation $+$ on the class $\hom_\mathscr{A}(X, Y)$ 
			as the composition
				\[A\overset{d}\longrightarrow A\oplus A \overset{f\oplus g}{\longrightarrow}
				B\oplus
				B\overset{d'}\longrightarrow B\label{op} \] 
			Where $d$ is the diagonal map, $f, g$ is the map given in lemma
			\ref{fg},  and $d'$ is the codiagonal map.
		\item For every object $A$ in the category $\mathscr A$, there is a
			morphism $-1_A$ such that $1_A + (-1_A) = 0_A$, where $1_A$ is the
			identity map on $A$, and $0_A$ is the  morphism that factors through 0 (unique, as 0 is initial and
			final). This condition ensures that $\hom(A, B)$ is an
			abelian group.\\
	\end{enumerate}
\end{definition}
We will show that composition of morphisms distributes over the operation
defined by equation \eqref{op}, and that the first two conditions give
$\hom(A, B)$ under \eqref{op} the structure of a commutative monoid. From this, we will demonstrate
that condition 3. suffices to turn the hom-classes into abelian groups. Before
we can embark on the proof of this fact, however, we must prove some results
about the natures of the objects which we seek to work with.

	In any ring, the zero object (the identity of the additive group) satisfies
$0a = 0b$ for any two objects $a$ and $b$. We will see that the zero object of a
category reflects this property, and indeed, does so uniquely. This fact will be
useful in our endeavor to demonstrate that zero is an additive identity.
\begin{lemma}\label{zero}
	The zero morphism $0:B\to B$ "equalizes any two morphisms $f$ and $g$" -
	that is, $0\circ f = 0\circ g$ for any
	$f, g:A\to B$, for any $B$, and is the unique morphism from $B$ to $B$ which
	equalizes any two morphisms from any $A$ to $B$.
\end{lemma}
\begin{proof}
	For any two objects $A, B$ in $\mathscr A$, consider the unique morphism
	$0:B\to B$ which factors through the zero object. Let $f, g:A\to B$. Then
	$0\circ f$ is a morphism mapping $A$ to $B$ which factors through 0, and
	thus is unique in this respect. However, $0\circ g$ is another such
	morphism, so $0\circ f = 0\circ g$. Conversely, suppose $z:B\to B$ has the
	property $z\circ f = z\circ g$ for every $f, g:A\to B$, for any $A$. Then,
	in particular when $A = B$ and $f=1_B, g= 0:B\to B$ we have $z\circ 1_B = 0$
	and thus $z=0$. Thus $0$ is the unique morphism which equalizes every pair of
	morphisms with $B$ as their target.
\end{proof}
The lemma above finds use in the following proof, which in turn will be used to
prove that the zero morphism is the additive identity for the monoid
$(\hom_\mathscr{A}(A, B), +)$. Intuitively, we can think of this lemma as
stating that a projection followed by an inclusion is a morphism which retains
no information; the only information which retains no information about
morphisms composed with it is the zero morphism. 
\begin{lemma}\label{izero}
	For a biproduct $B\oplus B$ in an additive category $\mathscr A$ with
	canonical projections $\pi_1, \pi_2$ and inclusions $\iota_1, \iota_2$, we
	have $\pi_1\circ \iota_2 = \pi_2\circ \iota_1 = 0$. 
\end{lemma}
\begin{proof}
	Let any pair of morphisms $f, g:A\to B$ for arbitrary $A$. Consider the
	composition $\pi_1\circ \iota_2\circ f$. By lemma \ref{fg'} we can rewrite
	this as $\pi_1\circ g\oplus f\circ \iota_2$. But by lemma \ref{fg} we can
	rewrite this again as $g\circ \pi_1\circ \iota_2$, which we can rewrite
	again by lemma $\ref{fg}$ applied once more as $\pi_1\circ g\oplus g\circ
	\iota_2$, which by lemma \ref{fg'} equals $\pi_1\circ\iota_2\circ g$. Thus
	$\pi_1\circ \iota_2 \circ f = \pi_1\circ\iota_2\circ g$, and so $\pi_1\circ
	\iota_2:B\to B$ equalizes any pair of morphisms from any object $A$ into $B$. Thus
	by lemma \ref{zero}, $\pi_1\circ \iota_2 = 0$. A symmetric argument holds for
	$\pi_2\iota_1$.
\end{proof}
The following lemma, and its dual (recorded below for completeness) form the
foundation for proving that composition distributes over $+$. 
\begin{lemma}
	For any map $f:A\to B$ we have $(f\oplus f) \circ d = d\circ f$ where $(f\oplus f)$ is
	as in lemma \ref{fg} and $d$ is the diagonal map.\label{ff}
\end{lemma}
\begin{proof}
	From the definition of $(f\oplus f)$ we have that it is the unique morphism which
	satisfies $\pi_1\circ (f\oplus f) = \pi_2\circ (f\oplus f) = f\circ p_1$, and from the
	definition of the diagonal we have that it is a morphism which
	satisfies $\pi_1\circ d = \pi_2\circ d = 1_A$.  Thus we see that the
	composition $(f\oplus f)\circ d$ is a morphism with $\pi_1\circ (f\oplus f)\circ d
	= \pi_2\circ (f\oplus f)\circ d = f$. Substituting $A$ projecting
	via $f$ twice onto $B$ into the definition for the biproduct $B\oplus B$
	demonstrates that this morphism is unique with respect to this property.
	However, $\pi_1\circ d\circ f = \pi_2\circ d\circ f = f$ by definition of
	$d$. Thus $(f\oplus f)\circ d = d\circ f$.
\end{proof}
\begin{lemma}
	For any map $f:A\to B$ we have $d_B'\circ (f \oplus f)  = f\circ d_A'$,
	where $d'_A$ is the codiagonal $d': A\oplus A\to A$ and $d'_B$ the
	codiagonal $d':B\oplus B\to B$.\label{ff'}
\end{lemma}
\begin{proof}
	Dual to lemma \ref{ff} by lemma \ref{fg'}.
\end{proof}
We now have all the materials required to embark on the main proof of this
section. Although this result is mentioned offhand in \autocite{ha} and in the exercises of 
\autocite{maclane}, the proof thereof is the author's original work.
We begin by proving composition destributes over addition, which will assist us greatly in our
further proof. 
\begin{theorem}
	For objects $A, B$ in an additive category $\mathscr A$:
	\begin{enumerate}
		\item Composition of functions distributes over to $+$ defined
			by \eqref{op}.
		\item The class $\hom(A, B)$ has the structure of an abelian group with
			operation $+$.
	\end{enumerate}
\end{theorem}
\begin{proof}
The class $\hom(A, B)$ is closed under the operation defined by \eqref{op}, as
	categories are closed under composition and each of the maps in the
	composition exists for any two objects $A, B$ and any two morphisms $f, g:A
	\to B$. \\

 To show that composition distributes over addition, consider $(f+g)\circ h$ for $f,
	g: A\to B$, and $h:C\to A$.
		\[C\xto{h}A\xrightarrow{d}A\oplus A\xrightarrow{(f\oplus g)}B\oplus
		B\xrightarrow{d'}B\]
		By lemma \ref{ff} this is the same as 
		$$C\xrightarrow{d}C\oplus C\xto{(h\oplus h)}A\oplus
		A\xrightarrow{(f\oplus g)}B\oplus
		B\xrightarrow{d'}B$$
		Which is exactly
		$$C\xrightarrow{d}C\oplus C\xrightarrow{(f\circ h\oplus g\circ h)}B\oplus
		B\xrightarrow{d'}B$$
		Which is the definition of $f\circ h + g\circ h$. Thus $(f+g) \circ h =
		f\circ h+g\circ h$. 
Bilinearity in the other argument of the composition follows from a dual argument
invoking lemma \ref{ff'}.
	Next, suppose $f, g, h$ morphisms in $\hom(A, B)$. Consider the expression
	$(f+g) + h$, defined to be the composition:
		$$A\xto{d}A\oplus A\xto{(f+g)\oplus h}B\oplus B \xto{d'} B$$
		But $f+g$ is defined to be $d'\circ f\oplus g\circ d$, so we can rewrite
		as:
		$$A\xto{d}A\oplus A\xto{d\oplus 1_A}(A\oplus A)\oplus A \xto{(f\oplus
		g)\oplus h} (B\oplus B)\oplus B \xto {d', 1_B} B\oplus B \xto{d'} B$$
		But by lemma \ref{prod_assoc}, specifically the naturality of the
		isomorphism, this is equivalent to. 
		$$A\xto{d}A\oplus A\xto{1_A\oplus d}A\oplus (A\oplus A) \xto{f\oplus(
		g\oplus h)} B\oplus (B\oplus B) \xto {1_B, d'} B\oplus B \xto{d'} B$$
		Which is (by a symmetric argument to the above) equivalent to $f+(g+h)$,
		so + is associative.

	Now let $A$ be an object in $\mathscr{A}$. Consider the morphism $1_A\oplus
	0\circ d$. Note that, by lemma \ref{fg}, and the definition of $d$,
	$\pi_1\circ 1_A\oplus 0\circ d = 1_A$ and $\pi_2\circ 1_A\oplus 0 \circ d =
	0$. Moreover, by substituting $A$ into the definition for $A\oplus A$ with the
	maps $1_A$ and $0$, we obtain that this morphism must be unique. 
	But $\iota_1$ is another morphism with this property by lemma \ref{izero}, so $\iota_1 =
	1_A\oplus 0\circ d$. Symmetrically, $\iota_2=0\oplus 1_A\circ d$. Then
	$1_A+0 = d'\circ 1_A\oplus 0 \circ d = d' \circ \iota_1 = 1_A$, by
	definition. Symmetrically, $0+1_A = 1_A$. This together with distributivity 
	yield that $f+0=f$ for any $A, B$ and any $f: A\to B$, with 0 the unique
	zero morphism from $A$ to $B$. Associativity and identity give $\hom(A, B)$
	the structure of a monoid; it will suffice to show that this monoid
	contains inverses for each morphism, and that $+$ is commutative.

	Let $f:A\to B$ for arbitrary objects $A, B$, and consider $ -f \coloneq (-1_A)\circ f$.
	Then note that $f+(-f) = f\circ(1_A+(-1_A)) = f\circ 0 = 0$ by distributivity.
	This gives each $\hom$-class the structure of a group with operation
	$+$, which composition distributes over. Finally, consider that
	if $f, g \in \hom(A, B)$ we have $-1_B\circ(f+g) = -1_B\circ f+-1_B\circ g =
	-f+-g$, but also that $-1_B\circ (f+g)$ is the unique inverse of $f+g$,
	which can also be written as $-g+-f$. Thus $-g+-f = -f+-g$ for every $f, g$, and
	$\hom(A, B)$ is an abelian group.
\end{proof}
Perhaps not obviously, this the operation defined by \eqref{op} is the only
operation on morphisms satisfying these properties. We direct the reader to
Chapter 8.II of
\autocite{maclane} for a
detailed proof of this fact, recording the theorem for completeness.

\begin{theorem}\label{conv}
	Any addition of morphisms over which composition distributes is given by the
	addition defined in
	\eqref{op}.
\end{theorem}


\subsubsection{Additive Functors}
As we have now fully defined a particular type of categories, the necessary
rhythm of a mathematics text dictates we now define a corresponding type of
functors. For us, these functors take the form of \textbf{additive functors},
and (as we will see) induce group homomorphisms on the morphisms in an additive
category.
\begin{definition}
	An \textbf{additive functor} is one which preserves biproducts.
\end{definition}
This somewhat minimal definition contains all the information we require to
construct a group homomorphism; intuitively, since we constructed our group
operation using only the biproduct, this is reasonable. 

\begin{lemma}\label{adconv}
	Consider additive categories $\mathscr A$ and $\mathscr B$ and functor
	$F:\mathscr A\to \mathscr B$. For morphisms $f$ and $g$, $F(f+g) =
	F(f)+F(g)$ if and only if $F$ is an additive functor.
\end{lemma}
\begin{proof}
	Consider $F(f+g) = F(d'\circ (f\oplus g)\circ d) = F(d')\circ F(f\oplus g)
	\circ F(d)$. But since $F(A\oplus A) = F(A)\oplus F(A)$, we must have that
	$F(d') = \hat{d'}$ where $\hat{d'}:F(B)\to F(B)\oplus F(B)$ is the
	appropriate codiagonal, as if $\hat\iota_i$ is the $i$th projection off the
	biproduct $F(B)\oplus F(B)$, we have that $\hat{d'}\circ
	\hat{\iota_1} =F(d')\circ F(\iota_1) =  F(d'\circ \iota_1) = 1_B$, and
	similarly $\hat{d'}\circ
	\hat{\iota_2} =F(d')\circ F(\iota_2) =  F(d'\circ \iota_2) = 1_B$. Dually,
	$F(d)$ is the diagonal map. Finally, a similar argument says $F(f\oplus
	g):A\oplus A\to B\oplus B$ is a morphism with $F(f\oplus g)\circ
	\hat{\iota_i} = F(f\oplus g)\circ F(\iota_i) = F(f\oplus g\circ \iota_i) =
	F(f)$ or $F(g)$ for $i=1, 2$ respectively. This means $F(f\oplus g) =
	F(f)\oplus F(g)$, so $F(f+g) =F(d')\circ F(f\oplus g) \circ F(d) =
	\hat{d'}\circ F(f)\oplus F(g)\circ\hat{d} = F(f) +  F(g)$, as intended.

	Conversely, suppose $F(f+g) = F(f) + F(g)$ for any suitable morphisms $f, g$. Consider
	$F(A\oplus B)$ with any $C$ projecting into $F(A)$ and $F(B)$
	via $f_1$ and $f_2$. Clearly $F(\pi_1)$ and $F(\pi_2)$ are morphisms from
	$F(A\oplus B)$ to $F(A)$ and $F(B)$ respectively. 
	Then $ F(\pi_1)\circ
	 F(	\iota_1)\circ f_1 + 
	F(\pi_1)\circ
	 F(	\iota_2)\circ f_2 =  
	 F(	\pi_1\circ\iota_1)\circ f_1
	+F(	\pi_1\circ\iota_2)\circ f_2=F(1)\circ f_1 + F(0)\circ f_2 =
	F(0)\circ f_1 = F(f_1)$. A symmetric argument holds for $f_2$, and a dual
	argument holds for inclusions. Finally, $F(\pi_i)\circ F(\iota_i) =
	F(\pi_i\circ\iota_i) = F(1) = 1$, so $F(A\oplus B)$ is a biproduct of $F(A)$
	and $F(B)$ with
	projections $F(\pi_1)$ and $F(\pi_2)$ and inclusions $F(\iota_1)$ and
	$F(\iota_2)$. Thus $F(A\oplus B)$ is a biproduct of $F(A)$ and $F(B)$, and
	so $F$ preserves biproducts.  
\end{proof}

\subsection{Abelian Categories}
We extend the notion of an additive category to capture more of the information
we require from an abelian group. In particular, we begin by extending the
notion of a kernel of a morphism. In a group, we are blessed with the ability to
consider the kernel as a subset of the domain of a morphism; however, in a
general category we have no such luxury. Indeed, the only way to define a kernel
elegantly is by specifying it through a universal property pertaining to a
certain morphism. The abstract nature of this construction, however, is rewarded
by the ability to dualize it to construct a \textbf{cokernel}. We will come to
understand that this cokernel encodes some intrinsic information about
quotient objects, allowing for a very neat extension of many set-theoretic
statements of algebra.
\begin{definition}
	Consider a morphism $f:A\to B$ in a category
	$\mathscr A$ with
	a zero object. A \textbf{kernel} $\ker f: \Ker f\to A$ of $f$, if it exists, is a
	morphism such that $f\circ \ker f = 0$ and
	every morphism $g$ with $f\circ g = 0$ factors uniquely as $\ker f\circ a$.
	A \textbf{cokernel} $\cok f:B\to \Cok f$, if it exists, is a morphism with $\cok f\circ
	f = 0$, and for any other $g$ with $g\circ f = 0$, $g$ factors uniquely as
	$a\circ \cok f$
\end{definition}

For now, we refer to "a kernel"; more development is necessary to
demonstrate that it is unique up to isomorphism. 
It is important to note that we associate the object $\Ker f$ with any kernel 
$\ker f$, so that the "kernel" we generally consider is effectively a pair $(\ker f, \Ker f)$.
Of course, the expression $\ker f$ alone captures this information; the source and
target of each morphism is data given in the morphism itself. 

To fully extend the notion of a group, we will
need analogous qualifiers to "injective" and "surjective". These will allow us
to translate theorems proven in the language of sets into the language of
categories. These come in the
form of the terms \textbf{monic} and \textbf{epi}:
\begin{definition}
	A morphism $m:A\to B$ in a category $\mathscr A$ is \textbf{monic} if for any two
	morphisms $f, g:A\to B$, whenever $m\circ f = m\circ g$ we can deduce that
	$f = g$ - that is, when $m$ can be "canceled on the left". Similarly, a
	morphism $e:A\to B$ in $\mathscr A$ is \textbf{epi} if for $f, g:A\to B$ we
	have $f\circ e = g\circ e\implies f=g$, that is, $e$ can be canceled on the
	right.
\end{definition}
When a category has a zero object, there is an alternate, equivalent definition
of monics and epis in terms of which morphisms they send to zero. Aside from
providing useful insight into how monics and epis function, this equivalent
definition is often useful when working with kernels and cokernels, and will
ultimately allow use the notion of an "exact sequence" to fully characterize
monics and epis. Furthermore, this lemma continues to demonstrate that monics generalize
injective homomorphisms - a homomorphism is injective if and only if it sends only the zero object
to zero. 
\begin{lemma}\label{monica}
	In an additive category, a morphism $m$ is monic if and only if $m\circ f =
	0\implies f = 0$, and a morphism $e$ is epi if and only if $f\circ e = 0
	\implies f = 0$.
\end{lemma}
\begin{proof}
	First suppose $m$ is monic. Then $m\circ f = 0 = m\circ 0\implies f = 0$ by
	definition of monic. Conversely, suppose $m$ has the property $m\circ f =
	0\implies f=0$ for any suitably composable $f$. Then consider $g, h$ with
	$m\circ h = m\circ g$. Let $f=h-g$. Then $m\circ h - m\circ g = 0$ by the
	definition of $-m\circ g$. However, $ 0 = m\circ h - m\circ g = m\circ(h-g)$
	(by bilenarity). Our assumption then implies
	that $h-g=0$. Then $h=g$ by adding $g$ to both sides, and $m$ is
	monic. A symmetric argument holds for epis.
\end{proof}

\begin{corollary}\label{monicaa}
	A morphism $m$ is monic if and only if every kernel of it is zero. A morphism $e$ is
	epi if and only if every cokernel of it is zero.
\end{corollary}
\begin{proof}
	Let $m$ be a monic. Clearly, $m\circ 0 = 0$; it will suffice to show that no
	nonzero morphism $g$ satisfies $m\circ g = 0$, but this follows from lemma
	\ref{monica}. Conversely, suppose $m$ has kernel 0. Then any morphism $f$
	with $m\circ f = 0$ factors through 0, and is therefore 0 by lemma
	\ref{monica}. A dual argument proves the lemma for epis.
\end{proof}
\begin{corollary}
	\label{isoker}
	The kernel and cokernel of every isomorphism is 0.
\end{corollary}

As we are working with kernels and cokernels often, it is useful to manipulate
them, almost algebraically. We will use the following lemma frequently to assist
in translation of theorems from set-theoretic kernels to categorical kernels. Intuitively,
it establishes that every kernel of $f$ "includes" into the source of $f$ by way
of our definition of monics;
remarkably, this is given exclusively by the universal property of kernels.
A powerful element of our
definition for abelian categories will be inverting this conditional, so that
every monic is indeed the kernel of some morphism, and every epi a cokernel.

\begin{lemma}\label{kermon}
	In an additive category, every kernel is monic, and every cokernel is epi.
\end{lemma}
\begin{proof}
	Let $f:A\to B$ be a morphism in an additive category $\mathscr A$,
	and let $f$ have a kernel $k:K\to A$. Suppose $k$ is not monic. Then there
	exist $g$ and $g'$ such that $k\circ g = k\circ g'$ with $g\neq g'$.
	However, $f\circ (k\circ g) = (f\circ k)\circ g = 0\circ g=0$  and $f\circ
	(k\circ g') = (f\circ k)\circ g' = 0\circ g' = 0$, so $k\circ g=k\circ g'$
	is a morphism which sends $g$ to zero and which factors through $k$ in two
	ways, contradicting the definition of a kernel. A symmetric argument holds
	for cokernels.
\end{proof}

The following lemma ensures that it is meaningful to refer to "the kernel" of a
morphism, at least up to isomorphism. Although we have been referring to the
unique kernel till now, we have not used it's uniqueness in any proof (so this
argument is not circular). 

\begin{lemma}\label{unique}
	The kernel and cokernel are unique up to isomorphism.
\end{lemma}
\begin{proof}
	Suppose there are two kernels of a morphism $f$, $k$ and $k'$. Then $k$ must
	factor uniquely through $k'$ as $k = k'\circ i$, and $k'$ must factor
	uniquely through $k$ as $k' = k\circ i$. Substitution gives $k' = k'\circ
	i\circ i'$. But the kernel $k'$ is monic by \ref{kermon}, so we can cancel
	to obtain $i\circ i' = 1$. A symmetric argument shows the symmetric
	composition is the identity, so $i$ is an isomorphism. Thus the two kernels
	differ only by a factor of an isomorphism. A dual argument
	proves that the cokernel is unique up to isomorphism.
\end{proof}
\begin{corollary}
	The object $\Ker f$ and $\Cok f$ are unique up to isomorphism.
\end{corollary}
We can now define the main construction of this chapter, the \textbf{abelian
category}. This category will allow us to pursue variants of many proofs
designed for abelian groups, rings, fields, modules, vector spaces, and so
on. 

\begin{definition}
	An \textbf{abelian category} is an additive category $\mathscr A$ which satisfies the following criterion:
	\begin{enumerate}
		\item Every morphism in $\mathscr A$ has a kernel and a cokernel. 
		\item Every monic is a kernel, and every epi a cokernel.
	\end{enumerate}
\end{definition}

We will see that condition 2 effectively encodes the first isomorphism theorem into our
definition of abelian categories, by replacing subgroups with monics and
quotient groups with cokernels and epis. 

\begin{definition}
	A subobject is a monic; a quotient object, an epi.
\end{definition}

We often identify the subobject with it's source, and the quotient object with
it's target; this leads to the more intuitive notion of a subobject "including"
into an object. Furthermore, we often equate two subobjects which differ by a
factor of an isomorphism. There are times, however, when we wish to take the
quotient of two subobjects; for
this we will need the following definition. Note that the relation "factors
through" takes the place of "is a subobject of" in this definition; this is an
important theme that will recur. 

\begin{definition}
	If a monic $a$ factors through a monic $b$ as $a = b\circ f$, define the \textbf{quotient
	of subobjects} $b / a$ to be the quotient object $\cok f$.  
\end{definition}

e have established that factoring is an important relation; we now proceed to
prove a plethora of results allowing us to easily and conveniently work with
this relation. Firstly, we present a factorization lemma motivating our
definition of "image", which will succeed the lemma.

\begin{lemma}\label{factoring}
	Given a morphism $f:A\to B$ in an abelian category which factors as $m\circ
	e$ for a monic $m$ and epi $e$, then $m$ must factor through $\ker(\cok f)$ and $e$ must
	factor through $\cok(\ker f)$. Furthermore, any morphism which factors as
	$g = f\circ h$ factors through $\ker (\cok f)$ and $\cok(\ker h)$.
\end{lemma}
\begin{proof}
	Consider $\cok f\circ m\circ e = 0$. By lemma $\ref{monica}$, this means
	that $\cok f\circ m = 0$, which means that $m$ must factor through $\ker
	(\cok f)$. A symmetric argument shows $e$ factors through $\cok (\ker f)$.
	Suppose another morphism $g = f\circ h$; then $\cok f \circ g = \cok f\circ
	f\circ h = 0\circ h = 0$ so $g$ must factor through $\ker(\cok f)$. A
	symmetric argument holds for the symmetric case.
\end{proof}
\begin{corollary}
	Any morphism $f$ in an abelian category factors through both $\ker(\cok f) $
	and $\cok (\ker f)$
\end{corollary}
\begin{proof}
	Consider the factorization $f = f\circ 1$ and $f=1\circ f$, where $1$ is an
	identity.
\end{proof}
It is perhaps intuitive that the image should be represented by our
generalized categorical subobject, and perhaps intuitive again that $f$ should
factor through a representative of this subobject. This intuition is captured in
the following definition. The final result of this section, a stronger version
of lemma \ref{factoring}, will build on results developed from lemma \ref{factoring}.

\begin{definition}
	 Consider a morphism $f:A\to B$ in an abelian category
	$\mathscr A$. Then the subobject $\ker\cok f$ is the
	\textbf{image} of $f$, denoted $\im f$. 
\end{definition}

Note further that $\im f: \Ker(\cok f)\to B$; we denote
$\Ker (\cok f)$ as  $\ima f$. Uniqueness (up to isomorphism) is given by the
uniqueness of the kernel. We will refer to $\im (f)$ and $\ker(\cok f)$
interchangeably, often working with $\ker(\cok f)$ for ease of manipulation and
understanding. 

The dual definition to this is also useful;

\begin{definition}
	 Consider a morphism $f:A\to B$ in an abelian category
	$\mathscr A$. Then the quotient object $\cok\ker f$ is the
	\textbf{coimage} of $f$, denoted $\coim g$.
\end{definition}

The following lemma is a useful technical result to relate images and
kernels with coimages and cokernels. It will become especially useful during our discussion
of exact sequences, where relating the kernel and image is of vital importance.
The lemma also plays a crucial role in our final theorem of this section, the
factorization of each morphism in an abelian category.
\begin{lemma}\label{differ}
	Let $f:A\to B$ be a morphism in an abelian category. Then
	$\ker(\cok(\ker f))$ and $\ker f$  differ by an isomorphism, as do
	$\cok(\ker(\cok f))$ and $\cok f$.
\end{lemma}
\begin{proof}
	Consider that $\ker f$ is a monic, and by the corollary of lemma
	\ref{factoring} factors through $\ker(\cok(\ker f))$ uniquely. Conversely,
	$f\circ \ker f = 0$, so $f$ factors uniquely as $f=g\circ \cok(\ker f)$ by the
	definition of cokernel. But then $f\circ \ker(\cok(\ker f)) = g\circ
	\cok(\ker f)\circ \ker(\cok(\ker f)) = g\circ 0 = 0$, so by the definition
	of kernel $\ker(\cok(\ker f))$ factors uniquely through $\ker f$. Thus $\ker
	f$ and $\ker(\cok (\ker f))$ differ by an isomorphism. A symmetric argument
	shows $\cok(\ker(\cok f))$ differs by an isomorphism from $\cok f$. 
\end{proof}
\begin{corollary}
	In an abelian category, any monic $m$ differs by an isomorphism from its image, and any epi
	$e$ differs by an isomorphism from its coimage.
\end{corollary}
\begin{proof}
	Every monic is the kernel of some $f$, so $m = \ker f$ differs by an
	isomorphism from $\ker(\cok(\ker f)) = \ker(\cok m)$. A symmetric argument holds for epis.
\end{proof}
	%
	%
The final lemma in this section gives perhaps the strongest characterization of
a general morphism possible in an abelian category. 
\begin{lemma}\label{me}
	Given a morphism $f:A\to B$ in an abelian category, $f = m\circ e$, with $m
	= \im f$ monic and $e = \coim f$ epi.
\end{lemma}
\begin{proof}
	Note that $\cok f\circ f = 0$, but by the definition of $\ker(\cok f)$, any
	morphism $g$ with $\cok f\circ g = 0$ must factor uniquely through $\ker(\cok
	f)$, so $f$ must factor through $\ker(\cok f)$. Thus $f = m\circ e$ for $m =
	\ker(\cok f)$, and $e$ some unique morphism. A tedious proof from
	\autocite{maclane} (Chapter 8, Section 1, Lemma 1) shows that $e$ must be
	epi.\footnote{When referencing this lemma, note that for a pair of morphisms
	$f, g:A\to B$, the equalizer $e$ such that $ea = eb$ is exactly the kernel
	$k$
	of the morphism $f-g$, as the universal morphism which gives $k\circ (f-g) =
	0\implies k\circ f-k\circ g = 0\implies k\circ f = k\circ g$, and so any
	abelian category has equalizers for every pair of suitable morphisms.}
	Since $m$ is monic, $f\circ t = 0$ if and only if $e\circ t = 0$ for any
	appropriate $t$. Thus $\ker f = \ker e$, as any morphism $t$ which sends $f$
	to zero also sends $e$ to zero.
	Then since an epi $e$ is the cokernel of its kernel, $f = m\circ e = m\circ \cok(\ker
	e) = \cok (\ker f)$ and we are done.
\end{proof}
\begin{corollary}
	The object $\Ker(\cok f)$ is the object $\Cok(\ker f)$ up to isomorphism. 
\end{corollary}
	%

\section{Homology, Exact Sequences, and Right Derived Functors}
\subsection{Sequences and Homology}
We begin by constructing the main object we will work with, a \textbf{sequence}
of objects.
\begin{definition}
	A \textbf{sequence} is an ordered collection of objects  in
	an abelian category $\mathscr A$, together with a morphism from each object
	to it's successor called the \textbf{differentials}. The composition of
	successive differentials must also be 0.
\end{definition}
 The crucial element of this construction is that the composition of
successive maps is zero. This, ultimately, allows us to construct quotients;
intuitively, speaking in group- or set-theoretic terms, the image of one map is
then a subset of the kernel of the next. Since both the image and the kernel are
subgroups, we can then form the quotient of the two - which, ultimately, is the
homology group. From a categorical perspective, we must express this by way of
factoring, which the following lemma does neatly. 
\begin{lemma}\label{chain}
	Let $f, g$ be morphisms in an abelian category. If $f\circ g = 0$, then
	$\im g$ factors through the kernel of $f$ as $\im a'$, and $\coim f$ factors
	as $a'\circ \cok g$ for some $a'$.
\end{lemma}
\begin{proof}
	Note that if $f\circ g = 0$, then $f$ factors through $\cok g$ as $f =
	a\circ\cok g$ for some morphism $a$. Then $(a\circ \cok
	g)\circ \ker(\cok g) = 0$, so $\ker(\cok g)$ factors through $\ker(a\circ \cok
	g) = \ker f$ as $\ker(\cok g) = \ker f\circ a$.
	Similarly, note that $g$ factors as $g = \ker f\circ a'$. Then
	$\cok(\ker f)\circ (\ker f\circ a) = 0$, so $\cok(\ker f)$ factors through
	$\cok (\ker f\circ a) = \cok g$ as $\cok(\ker f) = a'\circ \cok g$. 
\end{proof}

We now distinguish between \textbf{chain complexes} and their dual,
\textbf{cochain complexes}. As
we build towards a construction of right derived functors, we will focus
primarily on cochain complexes. However, every construction we perform can be
dualized by passing into the opposite category to apply to chain complexes. 

\begin{definition}
	If objects in the sequence are numbered such that the maps take object $n$ to
	object $n-1$, the sequence is called a chain complex. If maps take object
	$n$ to object $n+1$, the sequence is called a cochain complex.
\end{definition}

We will work primarily with cochain complexes, which we will denote by a capital
letter with a superscript bullet, $A^\bullet$. For a cochain complex
$A^\bullet$, the individual objects will be denoted $A^1,
A^2, $ etc. Having defined an object, it is now preferable to define a map between two such
objects.
\begin{definition}
	Let $A^\bullet$ and $B^\bullet$ be cochain complexes with differentials $d^n$
	and $\partial^n$. A cochain map $f$ is a $(f^n)_{n\in \mathbb
	Z}$ such that $f^n:A^n\to B^n$ and $f^n\circ d^n = \partial^n\circ f^{n-1}$.
\end{definition}

As we now have objects and morphisms, we can construct a category. What is
perhaps more interesting is that this category inherent the abelian properties
of the category from which it is constructed. This instantly admits the
interesting idea of higher order constructions, chain complexes of chain
complexes which (inductively) will all be abelian categories. Already, we see
the benefits of our general development; although it is beyond the scope of this
paper, we could instantly apply every result proven heretofore to the category
of cochain complexes of an abelian category, without any loss of rigor.

\begin{theorem}
	Given an abelian category $\mathscr{A}$, the category $\textbf{CCH}(\mathscr
	A)$ with objects cochain complexes of objects in $\mathscr A$ and morphisms
	cochain maps is an abelian category, with composition defined componentwise. 
\end{theorem}
\begin{proof}
	The categorical axioms (morphisms, objects, associativity, identity) follow
	from the categorical definitions in $\mathscr A$, using the fact that
	composition is defined componentwise. 
First we show that $\Kom (\mathscr A)$ is an additive category; that is, a category with
a \textbf{biproduct} $\oplus$ satisfying both the product and coproduct axioms,
which induces the structure of an abelian group on $\hom(X^\bullet, Y^\bullet)$
for any $X^\bullet, Y^\bullet$.
Note that for cochain complexes $X^\bullet$ and $Y^\bullet$ the product $(X\oplus
Y)^\bullet$ satisfies 

\centerline{
	\xymatrix{
		& &  X^n\\
		U \ar@/^1.1pc/[urr]\ar@/_1.1pc/[drr]\ar@{.>}[r]& X^n\oplus
		Y^n\ar[ur]^{\pi^n}\ar[dr]^{\pi^n}\\
		& &  Y^n
	}
}

For each $U$ that maps to $X^n$ and $Y^n$, and likewise satisfies

\centerline{
	\xymatrix{
		& &  X^n\ar@/_1.1pc/[dll]\ar[dl]_{\iota^n}\\
		U &\ar@{.>}[l] X^n\oplus Y^n\\
		& &  Y^n\ar@/^1.1pc/[ull]\ar[ul]_{\iota^n}
	}
}
For each $U$ which is mapped to by $X_n$ and $Y_n$,
by the definition of the biproduct for $X$. Since morphisms of cochain complexes
are just $\mathbb Z$-indexed families morphisms in $\mathscr A$ making ladders of cochain
complexes commute, and since the universal property of each of these diagrams
makes the ladders commute, we can see that the `indexwise' biproduct $(X\oplus
Y)^\bullet$ together with the projections and inclusions $\pi^\bullet,
\iota^\bullet$ satisfies the axioms for a biproduct, and so $\mathscr A^\bullet\oplus
B^\bullet$ exists. Furthermore, the sum induced on $\hom(X^\bullet, Y^\bullet)$ by $\oplus$ is
just the indexwise sum of morphisms in the category $\mathscr A$; the zero object is the
zero complex $0^n = 0$ for 0 the zero object in $\mathscr A$. Finally, if we let
$f^\bullet$ be a morphism of cochain complexes, then $f^n - f^n = 0$ for each $n$
because $\hom_\mathscr A(X^n, Y^n)$ has the structure of an abelian group;
as such, $f^\bullet:X^\bullet \to Y^\bullet$ has an additive inverse $-f^\bullet
= (-f)^\bullet$. Furthermore, this inverse commutes with the differential;
$-f^n\circ \delta - d\circ -f^{n-1}= -(f^n\circ \delta - d\circ f^{n-1})$ by
distributivity in the additive category $\mathscr A$. Then since $f^\bullet$
commutes with $d$, this is zero, so $-f^n\circ \delta = d\circ -f^{n-1}$ and $-f^\bullet$
lives in $\hom(X^\bullet, Y^\bullet)$. Thus, $\Kom(\mathscr A)$ is an additive
category. 

It suffices to show the existence of monic kernels and epi cokernels. Let
$f:A^\bullet\to B^\bullet$ be a morphism of cochain maps, and consider $\ker(f^n)$
for each $n$. By the definition of the kernel, we know that $f^n\circ \ker f^n = 0$,
and that if $g^\bullet$ is another cochain map, with $f^n\circ g^n = 0$, then
$g^n$ factors through $f^n$. It suffices to show that $\Ker f^n$ forms a cochain
complex. But we note that $d\circ \ker f^{n-1}$ (where $d$ is the differential
of $A^\bullet$) is a morphism satisfying $f^n\circ d\circ \ker f^{n-1} =
\ker f^{n} \circ f^{n} \circ d = 0$, and so factors through $\ker f^n$; 
this factorization yields a map $\partial^n:\Ker(f^{n-1})\to \Ker(f^n)$ so
that $\ker f^n$ commutes with $\partial$ by construction. Furthermore,
$\ker f^n\circ \partial^n\circ \partial^{n-1} =\ker f^{n-2} \circ d\circ d = 0$.
But since $\ker f^n$ is a monic, this means that $\partial ^n\circ \partial
^{n-1}$ is zero, and so is a differential making $\Ker(f^\bullet)$ a cochain
complex. Thus $\Kom(\mathscr A)$ has kernels; a dual argument shows cokernels. It
suffices to show that every monic is a kernel (and dualize for epis and
cokernels). Let $m^\bullet:A^\bullet \to B^\bullet$ be a monic in $\Kom(\mathscr A)$.
Then suppose $m^n\circ h = m^n\circ g$ for some $m^n$ and $h,g\in
\hom_\mathscr A(C^n, A^n)$. Then $m^n\circ (h-g) = 0$. But we can construct a
cochain complex:\\

\centerline{\xymatrix{
	...\ar[r]&0\ar[r]\ar[d]^{0}& C^n \ar[r]^{h-g}\ar[d]^{h-g} &A^n
	\ar[r]\ar[d]^{d} &0\ar[r]\ar[d]^{0} & ...\\
	...\ar[r]^d&A^{n-1}\ar[r]^d& A^{n}\ar[r]^{d}& A^{n+1}  \ar[r]^{d} &A^{n+2}\ar[r] & ...\\
}}

\bigskip

This is a cochain complex (the top row) with differential $\delta$ and a map $f$
for which the commutivity condition $\delta\circ f^n = d\circ (h-g) = 
\delta \circ f^{n-1}$. Since in this
composition $h-g$ commutes with the differential, composing with $m^n$ gives
the diagram \\

\centerline{\xymatrix{
	...\ar[r]&0\ar[r]\ar[d]^{0}& C^n \ar[r]^{h-g}\ar[d]^{h-g} &A^n
	\ar[r]\ar[d]^{d} &0\ar[r]\ar[d]^{0} & ...\\
	...\ar[r]^d&A^{n-1}\ar[r]^d\ar[d]^{m^{n-1}}& A^{n}\ar[r]^{d}\ar[d]^{m^n}&
	A^{n+1}  \ar[r]^{d}\ar[d]^{m^{n+1}} &A^{n+2}\ar[r]\ar[d]^{m^{n+2}} & ...\\
	...\ar[r]^d&B^{n-1}\ar[r]& B^{n}\ar[r]& B^{n+1}  \ar[r] &B^{n+2}\ar[r] & ...\\
}}

\bigskip

Which means $m^n\circ (h-g) = 0\implies h-g=0$, and so $m^n$ is a monic in $\mathscr
A$. But this means that it is the kernel of some morphism $f^n:C^n\to B^n$ for
every $n$. Since $\ker\cok\ker(f^n) = \ker(f^n)$,  we can suppose without loss
of generality that $f^n$ is epi. Furthermore, the universal property of kernels applied to the
composition $d\circ f^{n-1}$ where $d$ is the differential of $B^\bullet$ gives a
map $\delta ^n : C^{n-1} \to C^n$ such that $f^\bullet$ commutes with
$\delta^\bullet$. Moreover, we have $\delta^n\circ \delta^{n-1}\circ f^{n-2}=
f^{n}d^n\circ d^{n-1}$. Since $f$ is an epi, this means that $\delta^n\circ
\delta^{n-1}$ is zero. This gives a cochain complex $C^\bullet$ where $f^\bullet$ is a
cochain map between cochain complexes $C^\bullet$ and
$B^\bullet$. Since the kernel is defined degreewise, $m^\bullet$ is the kernel
of a cochain map $f^\bullet$. A dual argument shows that every epi is a cokernel, and we are
done.
\end{proof}

We now develop what is perhaps the main construction in this paper, the
\textbf{cohomology object}. If we were working in the category of abelian
groups, or the category of $R$-modules, the kernel of the differential would be
a subgroup or sub-module of the image, and we could take the quotient of the
two. As we are working in a general category, we replace the sets "image"  and
"kernel" with the morphisms $\ker(\cok d^n)$ and $\ker d^{n+1}$; we replace
"is a subset of" with "factors through", and we replace quotients with cokernels
of an appropriate morphism, or a quotient object of two subobjects.

\begin{definition}
	Let $A^\bullet$ be a cochain complex with differentials $d^n$. Then for each
	$d^n$, lemma \ref{chain} states that $\ker (\cok f)$ factors as $\ker( \cok
	f) = \ker f\circ a$. The \textbf{cohomology object} $H^n(A^\bullet)$ is the
	quotient object $\cok a = \ker d^{n+1}/\im d^n$
\end{definition}

Note the cokernel $\cok a$ effectively measures how much $\ker(\cok  d^n)$
differs from $\ker d^{n+1}$. Moreover, the cokernel in general is a sort of
"universal epi"; in set theoretic terms, a "universal surjection". As such, it
makes intuitive sense that it represents the quotient; it specifies the largest
object which sends a morphism to zero, capturing all information left after a
morphism has acted. We will see that the homology group is truly a
representation of the failure of a sequence to be \textbf{exact}.

\begin{definition} A
	sequence is \textbf{exact at an object} $B$ with morphisms $f:A\to B$ and
	$g:B\to C$ if the subobject associated with $\ker\cok f$ is equal to the
	subobject associated with  $\ker g$, and a sequence is said to be
	\textbf{exact} if it is exact at every object.  
\end{definition}
As hinted at above, we see that exactness represents when $\ker(\cok d^n)$ "is
effectively the same as" (differs by a factor of an isomorphism from) $\ker
d^{n+1}$; a homology group "measures" this failure. A simple way
to verify this assertion is to examine the effect of homology on an exact
sequence:
\begin{lemma}
	If a cochain complex is exact at $A^n$, the cohomology object $H^n(A^\bullet)$
	is the zero morphism.
\end{lemma}
\begin{proof}
	It will suffice to show that, for any two monics $f$ and $g$, if $f$ differs
	by an isomorphism from $g$,
	$f / g = 0$.  Note that if $f = g\circ i$ for an
	isomorphism $i$, by the definition of a quotient, $f/g = \cok i$.
	But $\cok i = 0$ by corollary \ref{isoker}, and we are done.
\end{proof}

We wish, ultimately, to turn cohomology into a functor.  Ultimately, given two cochain complexes $A^\bullet$ and
$B^\bullet$,
with a cochain map $f = (f^n)$ between them,
we seek to find morphisms from the $n$th cohomology of the first cochain map
to the $n$th cohomology of the second in a functorial manner. We go about this
by a construction drawing very heavily on one from
\cite{ruiter}. 

\begin{lemma}\label{homhom}
	Given cochain complexes $A^\bullet$ and $B^\bullet$ and a cochain map
	$f:A^\bullet \to B^\bullet$, we can construct the \textbf{cohomology map}
	$H^n(f):H^n(A^\bullet)\to H^n(B^\bullet)$.	
\end{lemma}
%
%

When examining the following proofs, it may be useful to refer to a diagram...
	\begin{figure}[h]
		\label{fig}
		\centering
		\caption{Constructing the Homology of a Morphism}
		\bigskip
		\centerline{
\xymatrix{
	\ima d^n\ar[r]^{a}\ar[dr]_{\im d^n}	&
	\Ker d^{n+1}\ar[r]^{\cok a}\ar[d]^{\ker d^{n+1}}&
	\Cok a& \ar@{=}[l] H^n(A^\bullet)				&
	\\
	A^{n-1}\ar[u]^{\coim d^n}\ar[d]^{f^{n-1}}\ar[r]_{d^n}	&
	A^{n}\ar[d]^{f^{n-1}}\ar[r]_{d^{n+1}}	&
		A^{n+1}\ar[d]^{f^{n-1}}\ar[r]_{d^{n+2}} & ... 	\\
	B^{n-1}\ar[d]_{\coim\partial^n}\ar[r]^{\partial^n}	&
	B^{n}\ar[r]^{\partial^{n+1}}	&	B^{n+1}	\ar[r]^{\partial^{n+2}} & ... 	\\
	\ima \partial^n\ar[r]^{a'}\ar[ur]^{\im \partial^n}	&
	\Ker \partial^{n+1}\ar[r]^{\cok a'}\ar[u]_{\ker \partial^{n+1}}&
	\Cok a'& \ar@{=}[l] H^n(B^\bullet)				&
	\\
}}
	\end{figure}

\begin{proof}
	Let $d^n$ denote
	the differential of $A^\bullet$, and $\partial^n$ the differential of
	$B^\bullet$. By lemma
	\ref{chain}, we know that $\im d^n$ factors as $\ker d^{n+1}\circ a$ for
		some $a:\ima d^n = \Cok (\ker d^n)\to \Ker d^{n+1}$. Similarly,
		$\im \partial^n$ factors as $\ker
	\partial^{n+1}\circ a'$ for  some $a':\Cok(\ker
	\partial^n)\to\Ker\partial^{n+1}$ by lemma \ref{chain}. In particular, note that $H^n(A^\bullet) =
	\Cok a$ and $H^n(B^\bullet) = \Cok a'$.

	We will begin by constructing maps $\alpha:\ker d^{n+1}\to\ker\partial^{n+1}$
	and $\beta:\ima d^n \to \ima \partial^n$ which commute with $a$ and
	$a'$. This will eventually allow us to construct our maps between cohomology
	objects. Note that the map $\partial^{n+1}\circ f^{n}\circ\ker d^{n+1}=
	f^{n+1}\circ d^{n+1}\circ \ker d^{n+1} = 0$. Thus, $f^n\circ \ker d^{n+1}$
	factors as $\ker \partial^{n+1}\circ \alpha$ where $\alpha:\ker d^{n+1}\to
		\ker\partial^{n+1}$. Similarly, note that
	the map $\cok \partial^n\circ\partial^n\circ f^{n-1}$ is zero identically.
	However, by the definition of a chain map this is equivalent to $\cok
	\partial^n\circ f^n\circ d^n$. Since $d^n$ factors as $\ker d^{n+1}\circ a \circ
		\coim d^n$ (by lemma \ref{me} and the definition of $a$; see figure
		\ref{fig} for a visual representation) we have that  
	$$\cok \partial^n\circ f^n\circ d^n = \cok \partial^n\circ f^n\circ \ker d^{n+1}\circ a \circ
	\coim d^n$$
	But since this whole expression was zero from it's construction, and $\coim d^n$ is epi, we have that $
	\cok \partial^n\circ f^n\circ \ker d^{n+1}\circ a$ is zero. Finally, this means
	that $f^n\circ \ker d^{n+1}\circ a$ factors as $\im f\circ \beta$. 
	Finally, we observe that $\ker \partial^{n+1}\circ \alpha
	\circ a = f^{n+1}\circ \alpha \circ a = \ker(\cok \partial^n)\circ \beta =
	\ker\partial^{n+1}\circ a'\circ\beta$. However, since $\ker\partial^{n+1}$ is a
	monic, we can cancel to obtain $a'\circ \beta = \alpha\circ a$. 
	
	This allows us to note that $\cok a'\circ \alpha\circ a = \cok a'\circ a'\circ
	\beta = 0$, so by the definition of a cokernel we establish that
	$\cok a'\circ \alpha$ factors uniquely as $h\circ \cok a$, where $h:\Cok a\to
		\Cok a'$. Define $H^n(f) = h$ to obtain the desired result. 
\end{proof}

We will refer back to this lemma frequently in our next one, demonstrating that
such a construction is sufficiently "nice" to demonstrate that cohomology is an
additive functor. 
The following theorem is also loosely based on proposition 3.1 in \cite{ruiter},
dualized for cohomology. 
\begin{theorem}\label{homfunc}
	$H^n:\textbf{CCH}(\mathscr A)\to A$ is an additive functor.
\end{theorem}
\begin{proof}
	First we must show that $H^n$ takes identity morphisms to identity
	morphisms. Consider $H^n(1)$ for an identity morphism $1$. Then, referring
	back to figure \ref{fig} and the definitions in lemma \ref{homhom}, we see that $\alpha$ in the definition of $H^1(1)$
	must be a morphism with $\ker d^{n+1} = \ker\partial^{n+1}\circ \alpha$. But
	in this case, $\ker\partial^{n+1} =\ker d^{n+1}$, so $\ker d^{n+1} = \ker
	d^{n+1}\circ \alpha$. But $\ker d^{n+1}$ is an epi; canceling gives us $1 =
	\alpha$. Then $\cok a'\circ \alpha = H^n(1)\circ \cok a$, so $\cok a' =
	H^n(1)\circ\cok a$. But since $d^n = \partial^n$, $a = a'$; since $\cok a =
	\cok a'$ is an epi, canceling gives $H^n(1) = 1_{H^n(A^\bullet)}$. 

	Now we must show that $H^n$ preserves composition. Consider $H^n(f\circ g)$,
	for cochain maps $f:A\to B$ and $g:B\to C$, denoting the differentials of $A,
	B$, and $C$ as $d, \delta,$ and $\partial$ respectively. Consider that
	$H^n(g\circ f)$ is the unique homomorphism with $\cok a''\circ
	\alpha_g\circ \alpha_f =
	H^n(f\circ g)\circ \cok a$, where $\alpha_f$ and $\alpha_g$ are the $\alpha$
	given in the proof of lemma \ref{homhom}, applied to $f$ and $g$ respectively.
	Similarly, $a:\ker(\cok d^n)\to \ker d^{n+1}$,  $a':\ker(\cok \delta^n)\to \ker
	\delta^{n+1}$, and $a'':\ker(\cok \partial^n)\to \ker \partial^{n+1}$, as in
	proof of lemma \ref{homhom}. But
	then $H^n(f)\circ H^n(g)\circ \cok a = H^n(f)\circ a'\circ \alpha_g =
	a''\circ \alpha_f\circ \alpha_g$, so $H^n(f\circ g) = H^n(f)\circ H^n(g)$.

	Finally, it will suffice to show $H^n(f+g) = H^n(f) + H^n(g)$. But if all
	morphisms are as in the proof of lemma \ref{homhom}, with $\alpha$ for $f$ and
	$\alpha'$ for $g$ (similarly for $\beta$), we have that $\cok a'\circ
	(\alpha + \alpha')\circ \cok(\ker d^n) = \cok a'\circ
	\alpha\circ \cok(\ker d^n) + \cok a'\circ
	\alpha'\circ \cok(\ker d^n)$ by bilenearity. Both of these are equal to
	$\cok a'\circ a'\circ (\beta + \beta') = 0$; as such, another application of
	distributivity yields that $\cok a'\circ \alpha + \cok a' \circ \alpha' = \cok
	a'\circ (\alpha + \alpha') = H^n(f+g)\circ
	\cok a$. A final application of distributivity shows that
	$(H^n(f)+H^n(g))\circ\cok a = H^n(f)\circ \cok a+H^n(g)\circ\cok a$, the
	definition of $H^n(f)$ simplifies this to $\cok a'\circ \alpha + \cok
	a'\circ \alpha'$, and the rest follows by uniqueness.	
\end{proof}

We have now fully defined cohomology as an additive functor, a very powerful
characterization and one that will lead directly to the construction of derived
functors, and the proof of several instances of their well-definedness. The next
tool that we will need to complete our definition is an injective resolution.
This will allow us to replace a single, complicated object with a simple
resolution which we understand well, a technique which forms the basis for
derived functors.

\subsection{Injective Resolutions}
\subsubsection{More on Exact Sequences}
Before we construct our injective resolutions, we will need a few more 
lemmas to prove further statements. The first of
these dualizes the definition of an exact sequence:
\begin{lemma}\label{coimage}
	If a sequence $A\xto f B\xto g C$ is exact at $B$, then $\cok f = a\circ
	\cok(\ker g)$
\end{lemma}
\begin{proof}
	By lemma \ref{chain}, $\cok(\ker g)$ factors through $\cok f$. By exactness,
	$\ker g = \ker(\cok f)\circ i$ for an isomorphism $i$; since $\cok f \circ
	\ker(\cok f)\circ i = 0$,
	we have that $\cok f$ must factor through $\cok(\ker(\cok f)\circ i)
	=\cok(\ker g)$ (and must do so as $a\circ \cok(\ker g)$).
\end{proof}
Finally, we give one more lemma allowing us to characterize monics and epis
using only exact sequences.
\begin{lemma}\label{insur}
	The sequence 
		$$0\xto 0 A\xto f B$$
		is exact if and only if $f$ is monic. The sequence
		$$A\xto f B\xto 00$$
		is exact if and only if $f$ is epi.
\end{lemma}
\begin{proof}
	Suppose $0\xto 0 A\xto f B$ is exact. Then $\ker\cok 0$ differs by an
	isomorphism from $\ker f$. Clearly the cokernel of 0 is the identity on $A$,
	as every morphism composed with zero is zero, and each morphism which
	factors through $A$ factors through $1_A$. But the kernel of an identity is
	0 by corollary \ref{isoker} thus the kernel of $f$ differs by an isomorphism
	from zero, and as such is zero. By corollary \ref{monicaa}, this means that $f$
	is a monic. Conversely, suppose $f$ is monic. Then the
	kernel of $f$ is zero by corollary \ref{monicaa}. But the cokernel of $0$ is $1_A$, as shown above, and
	the kernel of $1_A$ is zero by corollary \ref{isoker}. Thus, $\ker\cok 0 = 0
	= \ker f$, and the sequence is exact. A dual argument holds
	for epis.
\end{proof}
	%
	%
We now define projective and injective objects and resolutions. Although this definition may
seem obscure, it is necessary to build resolutions of objects in a way that
captures some representation their internal structure.
\begin{definition}
	A \textbf{projective} object $P$ is one which, for any epi $f:B\to C$ together
	with a map $\gamma:P\to Y$, there is a $\beta: P\to B$ such that $f\circ \beta
	= \gamma$. Similarly, an \textbf{injective} object $I$ is one which, for any
	monic $g:A\to B$ and any $\alpha: A\to I$ there is a $\beta:B\to I$ such
	that $\beta\circ g = \alpha$. An abelian category $\mathscr A$ has "enough
	projectives" if and only if for any object $A$ there is a epi $f:P\to
	A$ with $P$ projective. Similarly, $\mathscr A$ has enough injectives if and
	only if for any object $A$ there is a monic $f:A\to I$ with $I$
	injective. 
\end{definition}
\begin{definition}
		An injective (resp. projective) \textbf{resolution} of an object $A$ in an
		abelian category $\mathscr A$ is an exact cochain complex (resp. chain complex)
		$$0\to A\to I^0\to I^1 \to I^2\to I^3\to...$$
		$$\text{(respectively }...\to P_3 \to P_2\to P_1\to A\to 0 \text{ )}$$
		with each $I^i$ injective ($P_i$ projective).
\end{definition}

Given minimal conditions outlined in the above definitions, we can always form
injective resolutions. Eventually, this will guarantee the existence of the
derived functor; more generally, it permits a plethora of techniques (most of
which are beyond this paper) which can be used to simplify complicated objects.
\begin{theorem}
	If an abelian category $\mathscr A$ has enough injectives, we can form an
	injective resolution of any object $A$.
\end{theorem}
\begin{proof}
	Since $\mathscr A$ has enough injectives, there is a monic $d^0:A\to I^0$ for
	some $I^0$. Consider the cokernel $\cok d^0: I^0\to \Cok d^0$. This cokernel
	has a target $\Cok d^0$, which in turn admits a monic $a^0:\Cok d^0\to I^1$,
	as $\mathscr A$ has enough injectives. Define $d^1$ to be the composition
	$a^0\circ \cok d^0$, and proceed inductively. Supposing $d^n$ exists, let
	$a^n$ be the monic mapping $\Cok d^n$ to some injective which we denote
	$I^{n+1}$, as given by the "enough injectives" condition. Define
	$d^{n+1}:I^n\to I^{n+1}$ to be $d^{n+1} \coloneq a^n\circ \cok d^n$. This
	inductively defines an injective resolution:
		$$0\xto0 A\xto{d^0}I^0\xto{d^1}I^1\xto{d^2}I^2\xto{d^3} ...$$
	With each $I$ injective, and $d^{n+1}\circ d^n = a^n\circ \cok d^n \circ d^n
	= a^n\circ 0 = 0$. These conditions ensure that $A\to I^\bullet$ is a cochain complex. 
	Since $d^0$ is monic, lemma \ref{insur} guarantees that the sequence is
	exact at $A$. To show it is exact at each $I^n$, note that each
	$\im d^n$ factors through $\ker d^{n+1}$ by lemma \ref{chain}.
	Conversely, $d^{n+1} = a^n\circ \cok d^n$ with $a^n$ monic and $\cok d^n$
	epi, so $\cok d^n$ factors as $f\circ \coim d^{n+1}$ by lemma
	\ref{factoring}. But then $\cok
	d^n\circ \ker d^{n+1} = f\circ \coim d^{n+1}\circ \ker d^{n+1} = 0$,
	so $\ker d^{n+1}$ factors through $\im d^n$. Thus $\ker d^{n+1}$ and
	$\im d^n$ differ by an isomorphism, and so $\ker d^{n+1}$ differs by
	an isomorphism from $\im d^n$. Therefore $I^\bullet$ is an exact
	cochain complex of injective objects, or an injective resolution.
\end{proof}

A dual construction can be made for projective resolutions; however, this paper
does not concern them, so it is omitted for brevity. There is no guarantee that
this injective resolution is unique, which means it is
impossible to use it to define a functor. However, our next lemma provides a way
of relating two such injective resolutions. First, however, we must define a
certain equivalence of cochain maps, which will allow us to compare such maps.
Eventually, this equivalence will prove to induce equality under homology.

\begin{definition}
	Let $f,g:A^\bullet\to B^\bullet$, where $A^\bullet$ has differentials $d^n$
	and $B^\bullet$ has differentials $\partial$.
	Two cochain maps $f$ and $g$ are \textbf{homotopic} if  there is an $s^i: I^i\to J^{i-1}$ such that
	$f^n - g^n = \partial^n\circ s^{n-1} + s^{n}\circ d^{n+1}$.
\end{definition}

\begin{theorem}[Comparison Lemma, adopted from \cite{rotman}]\label{compare}
	For objects $A$ and $B$ in an abelian category with enough injectives, and
	complexes of injective resolutions $A\to I^\bullet$ and
	$B\to J^\bullet$, a map
	$f:A\to B$ induces a cochain map $f^i:I^i\to J^i$ with $f^n\circ d^n =
	\partial^{n}\circ f^{n-1}$ (where $d^n$ is the
	differential of $I^\bullet$ and $\partial^n$ the differential of
	$J^\bullet$). Any two cochain
	maps $f^i$ and $g^i$ induced by $f$ are homotopic. 
\end{theorem}
\begin{proof}
	For the base case $n=0$, the differentials $d^0:A\to I^0$ and $\partial^0:B\to J^0$
	are monic by lemma \ref{insur}, so applying the definition of an injective object to $d^0$ and the
	composition $\partial^0\circ f$ gives a morphism $f^0:I^0\to J^0$ with
	$f^0\circ d^0 = \partial^0\circ f$. We then construct
	each subsequent $f^n$ inductively, assuming $f^{i}$ for $-1 \leq i\leq n$
	are constructed such that $f^{n}\circ d^n = \partial^{n}\circ f^{n-1}$.

	Note that $\partial^{n+1}\circ f^{n}\circ d^{n} = \partial^{n+1}\circ
	\partial^n\circ f^{n-1}
	= 0$, so $(\partial^{n+1}\circ f^{n})$ factors through $\cok(d^{n})$,
	which by the exactness of the injective resolution and lemma \ref{chain} factors through
	$\cok(\ker(d^{n+1}))$. But then we can write $(\partial^{n+1}\circ
	f^n)$ as $\eta\circ \cok(\ker d^{n+1})$ for some morphism
	$\eta:\Ker(\cok d^{n+1})\to J^{n+1}$ by lemma \ref{coimage}. Applying the definition of the
	injective object $J^{n+1}$ with morphisms $\eta:\Ker(\cok d^{n+1})\to
	J^{n+1}$ and monic $\ker(\cok d^{n+1}):\Ker(\cok d^{n+1})\to I^{n+1}$ gives
	a morphism $\beta:I^{n+1}\to J^{n+1}$ with $\beta\circ \ker(\cok d^{n+1}) =
	\eta$, which (composing on the right) gives $\beta\circ \ker(\cok
	d^{n+1})\circ \cok(\ker d^{n+1}) = \beta\circ d^{n+1} = \eta\circ \cok(\ker
	d^{n+1}) = \partial^{n+1}\circ f^n$. Defining $f^{n+1}\coloneq \beta$ gives
	us our intended construction.

	It will thus suffice to show uniqueness up to homotopy equivalence. Let
	$f^{i}$ and $g^i$ be maps from $I^i$ to $J^i$, satisfying the necessary
	conditions. Construct terms of a map $s^i$ inductively. The base case can be
	shown easily by letting $s^{-1}$ and $s^0$ both be zero. Inductively, assume
	that $f^n - g^n = \partial^{n} \circ s^{n-1} + s^{n}\circ d^{n+1}$. 

	Note
	that $(f^{n+1} - g^{n+1} - \partial^{n+1}\circ s^{n})\circ d^{n+1} = (f^{n+1} -
	g^{n+1})\circ d^{n+1} - \partial^{n+1}\circ s^{n}\circ d^{n+1}$. By the
	definition of $s^{n}$ we can rewrite as $(f^{n+1} -
	g^{n+1})\circ d^{n+1} - \partial^{n+1}(f^n-g^n - \partial^{n}\circ s^{n-1})$.
	Expanding, we obtain $f^{n+1}\circ d^{n+1} -
	g^{n+1}\circ d^{n+1} - \partial^{n+1}\circ f^n+\partial^{n+1}\circ g^n +
	\partial^{n+1}\circ \partial^{n}\circ s^{n-1}$. The last term cancels, and
	we re-arange (by the fact $\hom(I^n, J^{n+1})$ is an abelian group) to obtain:

	$$f^{n+1}\circ d^{n+1} -\partial^{n+1}\circ f^n
	-g^{n+1}\circ d^{n+1}  +\partial^{n+1}\circ g^n$$
	But by the definition of a cochain map, $f^{n+1}\circ d^{n+1} =
	\partial^{n+1}\circ f^{n}$, and a similar equivalence holds for $g$, so the
	whole sum is zero. 
	Thus, $\alpha \coloneq f^{n+1} - g^{n+1} - \partial^{n+1}\circ s^{n}$
	factors by the definition of a cokernel as $\alpha = a\circ \cok d^{n+1}$. But by exactness
	and lemma \ref{coimage}, $\alpha = \eta \circ \cok \ker d^{n}$. We then
	apply the definition of an injective object, noting that $\eta:\Ker \cok d^{n+1} \to
	J^{n}$, and that $\ker(\cok d^{n+1}):\Ker \cok d^{n+1}\to I^{n+1}$ is an
	injection. Thus the definition of injectivity gives a morphism
	$\beta:I^{n+1}\to J^{n}$ with $\beta \circ \ker(\cok d^{n+1}) = \eta$.
	Composing $\cok(\ker d^{n+1})$ on the right gives us that
	$\beta\circ \ker(\cok d^{n+1})\circ \cok(\ker d^{n+1}) = \eta \circ \cok(\ker
	d^{n+1})$. But then lemma \ref{me} and the definition of $\eta$ gives us
	$\beta\circ d^{n+1} = f^{n+1} - g^{n+1} - \partial^{n+1}\circ s^{n}$, or that 
	$\beta\circ d^{n+1}+ \partial^n\circ s^{n-1} = f^{n} - g^{n}$. We thus
	define $s^n \coloneq \beta$, and we have constructed a cochain homotopy
	inductively.
\end{proof}

\subsection{Derived Functors}

This brings us to our final construction, the derived functor. Briefly, the
right derived functor is a method of "fixing" or "extending" certain functors
which take monics to monics, but do not always take epis to epis. However,
beyond this, the right derived functor is a method of characterizing many types
of cohomology; group cohomology or $\Ext(A, -)$, for example, is the right
derived functor of $\hom(A, -)$. A dual construction, the left derived functor,
is possible by passing into the opposite category and performing the same
construction; we will omit a rigorous proof.
\begin{definition}
	Choose and fix an injective resolution $I^\bullet$. Given an (additive)
	functor $F:\mathscr A\to \mathscr B$ where
	$\mathscr A$ is an abelian category with enough injectives and $\mathscr B$
	is an abelian category, the $i$th \textbf{right derived functor} $R^i F(A)$ for an
	object $A$ is given by $R^iF(A) = H^i(F(0\to A\to I^\bullet))$. The $i$th
	right derived functor likewise acts on a morphism by the rule $R^iF(f) =
	H^n(f^\bullet)$, where $f^\bullet$ is the cochain map induced by on injective
	resolutions $I^\bullet$ of $A$, and $J^\bullet$ of $B$.
\end{definition}
\begin{theorem}
	Given a left exact functor $F$ and a choice of injective resolution
	$I^\bullet$ for every object $A$, the right derived functor $R^i F$ obtained
	using $I^\bullet$ is an additive functor.
\end{theorem}
Referring to figure \ref{fig} frequently during the following proof is advised.
\begin{proof}
	First, we must show that $R^iF(f)$ is well defined. The comparison lemma
	states that any two cochain maps $f$ and $f'$ derived from a morphism $F(g)$
	for some $g:A\to B$ will
	be cochain homotopic; given a cochain map $f$ induced by a map $g$,
	the cochain map $F(f)$ will also be a cochain map over $F(g)$ as
	$F(f)^0\circ F(d^0) = F(f^0\circ d^0) = F(\partial^{n}\circ g) =
	F(\partial^n)\circ F(g)$. Thus, it suffices to show that $H^n(d^n\circ s^{n-1} +
	s^n\circ \partial^{n+1}) = H^n(d^n\circ s^{n-1}) + H^n(
	s^n\circ \partial^{n+1}) = 0$. But $\alpha =  a'\circ
	\cok(\ker \partial^n) \circ s^{n-1} \circ \ker d^{n+1}$ is a morphism
	satisfying $\ker \partial^{n+1}\circ\alpha = f^n\circ \ker d^{n+1} $, where
	$f^n = \partial^n\circ s^{n-1}$ (as seen in figure \ref{fig}). Then $H^n(\partial^n\circ s^{n-1})\circ
	\cok a = \cok a' \circ a'\circ
	\cok(\ker \partial^n) \circ s^{n-1} \circ \ker d^{n+1} = 0$. Since $\cok a$
	is an epi, lemma \ref{monica} states that $H^n(\partial^n\circ s^{n-1}) =
	0$. 

	Now consider an $\alpha$ as constructed in figure \ref{fig} corresponding to
	$s^n\circ d^{n+1}$. Noting that $\ker\partial^{n+1} \circ \alpha = 
	s^n\circ d^{n+1}\circ\ker d^{n+1} = 0$; by lemma \ref{monica}, this means
	$\alpha = 0$. But then $\cok a'\circ \alpha = H^n(s^n\circ d^{n+1}) \circ
	\cok a = 0$. Since $\cok a$ is epi, a final reference to lemma \ref{monica}
	gives us that $H^n(s^n\circ d^{n+1}) = 0$. Thus we have show than any map
	which is homotopic to zero has homology zero, and therefore shown that
	$R^iF(f)$ is well defined. 

	Now we must show that $R^i F$ is a functor. First, $R^iF(1) = H^i(i)$, where
	$i$ is homotopic to the identity. But by the above, this means that
	$R^iF(1)$ is the identity, as in tended. Similarly, given two composable
	morphisms $f$ and $g$, $R^iF(f\circ g) = H^i(F(f)'\circ F(g)')$, where
	$F(f)'$ and $F(g)'$ denote the chain maps over $F(f)$ and $F(g)$,
	respectively. However, theorem \ref{homfunc} shows that $H^i(F(f)'\circ
	F(g)') = H^i(F(f)')\circ H^i(F(g)')$. 

	Finally, we must show that $R^iF$ is additive. By theorem \ref{adconv}, it
	will suffice to show that for $f, g:A\to B$ morphisms, $R^iF(f+g) =
	R^iF(f)+R^iF(g)$. But since $F$ is additive, $R^iF(f+g) = H^i(F(f+g)') = H^i(F(f)'+F(g)')$ (where
	primes again denote passage to the the chain map over a morphism). Since
	homology is an additive functor, we then have $R^iF(f+g) =
	H^i(F(f)')+H^i(F(g)')$, as intended.
\end{proof}

It is tedious and inelegant to choose an injective resolution every time a
derived functor is to be computed; moreover, it makes computation difficult. Our
following theorem shows that this choice is immaterial, as computing the derived
functor twice with different injective resolutions yields two
isomorphic results. This completes our construction.

\begin{theorem}
	Let $R^nF(A)$ denote the right derived functor of an object $A$ computed using an injective resolution
	$I^\bullet$, and $\hat R^n F(A)$ denote the right derived functor of the
	same object computed with
	injective resolution $J^\bullet$. Then $R^n F (A) \cong \hat R^n F(A)$.
\end{theorem}
\begin{proof}
	The identity map on $A$ induces a cochain map $i$ from $J^\bullet$ to
	$I^\bullet$ and a cochain map $i^{-1}$ from $I^\bullet$ to $J^\bullet$. Then
	$i\circ i^{-1}$ is a cochain map from $I^\bullet$ to $I^\bullet$. Applying the
	functor $F$ gives a cochain map $F(i):F(I^\bullet)\to F(I^\bullet)$ induced
	by the identity $F(1_A) = 1_{F(A)}$ on $F(A)$, which is
	homotopic to the identity on $F(I^\bullet)$ (as the identity on
	$F(I^\bullet)$ is also a map induced by the
	identity of $F(A)$), so $R^nF(i)\circ R^nF(i^{-1}) = H^n(F(i)) =
	1_{F(I^\bullet)}$. A symmetric argument shows that $R^nF(i^{-1})\circ R^nF(i) =
	1_{F(J^\bullet)}$, and so $R^nF(i)$ is an isomorphism.
\end{proof}

This construction concludes our paper. As promised, we have omitted most
motivation and examples for the sake of brevity, directing the reader to
\cite{rotman} for a less categorical development with more concrete examples,
\cite{weibel} for a more comprehensive treatment of the homological implications
of derived functors, and \cite{maclane} for more examples and theorems regarding
abelian categories.

\nocite{*}
\printbibliography

\end{document}